\documentclass[english,letterpaper]{article}
\usepackage[english]{babel}
\usepackage{amsmath,amsthm,amssymb}
\usepackage{latexsym}
\usepackage{enumerate}
\usepackage[latin1]{inputenc}
\usepackage{layout}
\usepackage{type1cm}
\usepackage{hyperref}

\newtheorem{theorem}{Theorem}
\newtheorem{lemma}{Lemma}
\newtheorem{proposition}{Proposition}

\theoremstyle{remark}

\title{The Ruelle operator for symmetric $\beta$-shifts}

\author{Artur O. Lopes\thanks{IME - UFRGS. Partially supported by CNPq}\, and Victor Vargas\thanks{IME - UFRGS. Supported by PNPD-CAPES grant.} 
}

\begin{document}

\maketitle

\begin{abstract}
Consider $m \in \mathbb{N}$ and $\beta \in (1, m + 1]$. Assume that $a\in \mathbb{R}$ can be represented in base $\beta$ using a development in series $a = \sum^{\infty}_{n = 1}x(n)\beta^{-n}$ where the sequence $x = (x(n))_{n \in \mathbb{N}}$ take values in the alphabet $\mathcal{A}_m := \{0, \ldots, m\}$. The above expression is called the $\beta$-expansion of $a$ and it is not necessarily unique.  We are interested in  sequences $x = (x(n))_{n \in \mathbb{N}} \in \mathcal{A}_m^\mathbb{N}$  which are associated to all possible values $a$ which have a unique expansion. We denote the set of such $x$ (with some more  technical restrictions) by $X_{m,\beta} \subset\mathcal{A}_m^\mathbb{N}$. The space $X_{m, \beta}$ is called the symmetric $\beta$-shift associated to the pair $(m, \beta)$. It is invariant by the shift map but in general it is not a subshift of finite type.
Given  a H\"older continuous potential $A:X_{m, \beta} \to\mathbb{R}$, we consider the Ruelle operator $\mathcal{L}_A$  and we show the existence of a  positive eigenfunction $\psi_A$ and an eigenmeasure $\rho_A$ for some values of $m$ and $\beta$. We also consider a variational principle of pressure. Moreover, we prove that the family of  entropies $(h(\mu_{tA}))_{t>0}$ converges, when $t \to\infty$, to the maximal value among the set of all possible values of entropy of all $A$-maximizing probabilities.
\end{abstract}


\vspace*{2mm}


\vspace*{2mm}

\section{Introduction}

Statistical Mechanics and Thermodynamic Formalism are branches of mathematics which are interested in the study of properties of systems of particles defined on lattices, whose interactions are defined from potentials taking real  values. Usually it is assumed that the potential is at at least continuous. One of the main topics of interest  is the study of Gibbs states and equilibrium states. There are several papers concerning the existence and uniqueness of equilibrium probabilities  in the case the potential  is H\"older continuous. 
In the most part of these works the important ergodic properties are derived  from properties of the eigenprobabilities and eigenfunctions of the Ruelle operator associated to $A$.

The seminal work using this approach was presented by Ruelle in \cite{MR0234697} as an instrument for the study of thermodynamic properties of systems defined on uni-dimensional lattices. From this paper several important results were derived. Bowen, Sinai, Parry and Pollicott made  important contributions for compact finite type subshifts (see \cite{MR1085356} for a nice presentation of the theory) and they also get interesting results in number theory.  In more general cases, such as the case of  symbolic spaces with a countable number of spins, Mauldin and Urba\'nski in \cite{MR1853808} and Sarig in \cite{MR1738951}, made some very important contributions on the noncompact countable setting (see also \cite{BG}, \cite{BF} and \cite{MR3864383}). On other hand, in  \cite{MR3377291}, \cite{MR3538412} and \cite{Lea} are presented advances regarding Thermodynamic Formalism and problems of selection at zero temperature in the classical $XY$-model (the case where the set of spins is not countable) in a compact setting.

Consider $m \in \mathbb{N}$ and $\beta \in (1, m + 1]$, we are interested in representing a real number $a$ in base $\beta$ using a development in series of the form $a = \sum^{\infty}_{n = 1}x(n)\beta^{-n}$, where the sequence $x = (x(n))_{n \in \mathbb{N}}$ take values in the alphabet $\mathcal{A}_m := \{0, \ldots, m\}$. The above expression is called the {\bf $\beta$-expansion of $a$} which is obviously not necessarily unique (see \cite{MR1153488}). 

From now on, we  denote by $\mathcal{U}_m$ for the set of real numbers $\beta$ belonging to the interval $(1, m + 1]$ for which the number  $1$ has a unique $\beta$-expansion - this set is known as {\bf set of univoque bases}. 

We are going to use the notation $\mathcal{U}_{m, \beta}$ for denoting the set of numbers in the interval $\left[0,\frac{m}{\beta - 1}\right]$ that have a {\bf unique $\beta$-expansion}, with $\beta \in (1, m + 1]$.  

In general terms, we are interested here, for a fixed value of $\beta\in \mathcal{U}_m$, on the strings 
$x = (x(n))_{n \in \mathbb{N}} \in \mathcal{A}_m^\mathbb{N}$ which are obtained from values $a$ which have an unique $\beta$ expansion (we will be more precise in a moment). We will consider here the action of the shift in the set of such strings $x$ and the corresponding  ergodic  properties of shift invariant probabilities.

The dynamical systems  which we consider here is widely known in the literature as the symmetric $\beta$-shifts. They were introduced by Sidorov in \cite{MR1851269} as a generalization of the classical $\beta$-shifts (see \cite{MR2180243} and \cite{MR0466492}) for the case $\beta \in (1, 2)$. 

The $\beta$-expansions of real numbers  is one of the main topics of interest in number theory (classical contributions  were presented by Erd\"os, Parry and R\'enyi in \cite{MR1078082}, \cite{MR0142719}, and \cite{MR0097374}. 

In \cite{MR3223814} it is studied the topological properties of symmetric $\beta$-shifts when $\beta \in (1, 2)$. In the mentioned  work are presented topological properties of these subshifts, characterizing them from the behavior of the $\beta$-expansions. In a more recent work \cite{MR3896110} Alcaraz et al. extended these results for a more generalized context, when $\beta \in (1, m + 1]$, for some arbitrary natural number $m$.

Our aim here is to understand the Thermodynamic Formalism for this model. A special analysis  (see section \ref{preliminaries-section}) will be required in order to show that the Ruelle operator is well defined (we have to take care of the preimages of a given point $x$ on the shift space in order to get a local homeomorphism).
 
More precisely,  we are interested in study of properties of the Ruelle operator defined in the context of symmetric $\beta$-shifts: existence and uniqueness of Gibbs states associated to a H\"older continuous potential.
\medskip

First we will present the definition of the shift space (it is not a shift of finite type) which will be the main focus of our paper
 
Given $m$ the {\bf generalized golden ratio} is defined as
\begin{equation*}
\mathcal{G}(m) =
\begin{cases}
k + 1 &, m = 2k \\
\frac{k + 1 + \sqrt{k^2 + 6k + 5}}{2} &, m = 2k + 1 
\end{cases} 
\end{equation*}

This number satisfies $\mathcal{U}_{m, \beta} \neq \emptyset$ for any $\beta \in (\mathcal{G}(m), m + 1]$ and $\mathcal{U}_{m, \beta} = \emptyset$ for each $\beta \in (1, \mathcal{G}(m))$ (see for instance \cite{MR3653101}). 

We will define the symmetric $\beta$-shifts from the sets $\mathcal{U}_m$ and $\mathcal{U}_{m, \beta}$ (with some extra conditions) for values $\beta \in (\mathcal{G}(m),m+1]$. This setting was introduced in \cite{MR3570134} in a work mainly interested in the study of topological properties of univoque sets. 

Set $X_m := (\mathcal{A}_m)^{\mathbb{N}}$ equipped with the usual lexicographic order $\prec$, which is defined as $x \prec y$, if and only if, there exists $n \in \mathbb{N}$, such that, $x(j) = y(j)$, for all $j < n$ and $x(n) < y(n)$. 

Set $\sigma: X_m \to X_m$ the shift map defined by $\sigma((x(n))_{n \in \mathbb{N}}) = (x(n+1))_{n \in \mathbb{N}}$. Henceforth, we will denote by $\overline{x} = (\overline{x(n)})_{n \in \mathbb{N}}$ its corresponding reflection, that is, $\overline{x(n)} = m - x(n)$, for each $n \in \mathbb{N}$. Besides that, for any finite word $\omega = \omega(1) \ldots \omega(l)$ we will define its reflection as $\overline{\omega} = (m - \omega(1)) \ldots (m - \omega(l))$. We will also use the following notation: $\omega^+ = \omega(1) \ldots (\omega(l) + 1)$ and, finally,  $\omega^{\infty}$, for the periodic sequence $(x(n))_{n \in \mathbb{N}}$ satisfying $x(kl + i) = \omega(i)$, for each $k \in \mathbb{N} \cup \{0\}$ and any $i \in \{1, \ldots, l\}$. 

Note that the definition of $\omega^+$ just make sense in the case that $\omega(l) \neq m$, and the definition of $\omega^-$ just make  sense in the case that $\omega(l) \neq 0$.

We will say that a sequence $x$ is infinite, if it has not a tail of the form $0^{\infty}$, in other case, we will say that the sequence $x$ is finite. 

We name {\bf greedy (resp. lazy) $\beta$-expansion} of a number $a \in \left[0, \frac{m}{\beta - 1}\right]$, to the largest (resp. smallest) sequence, regarding the lexicographic order, in the set of all possible $\beta$-expansions of $a$, which can be either, a finite or an infinite sequence. Observe that both of these sequences lazy and greedy agree, if and only if, the real number $a$ has a unique $\beta$-expansion. 

We are going to name {\bf quasi-greedy (resp. quasi-lazy) $\beta$-expansion} of a number $a \in \left[0, \frac{m}{\beta - 1}\right]$, to the largest (resp. smallest) infinite sequence, with respect to the lexicographic order, in the set of all possible $\beta$-expansions of $a$. Note that when the greedy (resp. lazy) $\beta$-expansion of a number $a$ is infinite, it agrees with the quasi-greedy (resp. quasi-lazy) $\beta$-expansion of the number $a$. 

A typical example of the above definitions is the following. Taking $\beta = \frac{1 + \sqrt{5}}{2}$, we have that the lazy $\beta$-expansion of $1$ is $x = 01^{\infty}$ and the greedy $\beta$-expansion of $1$ is $x = 110^{\infty}$. Furthermore, in this example the quasi-lazy $\beta$-expansion of $1$ coincides with the lazy $\beta$-expansion of $1$ because it is infinite and the quasi-greedy $\beta$-expansion of $1$ is $x = (10)^{\infty}$. 

Set $x^{m, \beta}$ the quasi-greedy $\beta$-expansion of $1$, from the greedy algorithm, it is easy to verify that $\overline{x^{m, \beta}}$ is the quasi-lazy $\beta$-expansion of $\frac{m}{\beta - 1} - 1$ when $\beta \in (\frac{m}{2} + 1, m + 1]$. 

Fixing $\beta \in (\mathcal{G}(m), m + 1]$, we define the set $\mathcal{W}_{m, \beta} = \mathcal{U}_{m, \beta} \cap \left(\frac{m}{\beta - 1} - 1, 1\right)$, and $\pi_{m, \beta} : X_m \to \left[0, \frac{m}{\beta - 1}\right]$ as the map assigning the real number $\sum^{\infty}_{n = 1}x(n)\beta^{-n}$ to each sequence $(x(n))_{n \in \mathbb{N}} \in X_m$. It is also easy to verify that 
\[
\pi_{m, \beta}^{-1}(\mathcal{W}_{m, \beta}) = \{x \in X_m : \overline{x^{m, \beta}} \prec \sigma^k x \prec x^{m, \beta}, \forall k \in \mathbb{N} \cup \{0\}\} \,.
\]

From the above, we can define the {\bf symmetric $\beta$-shift} associated to the pair $(m, \beta)$ as the $\sigma$-invariant set  
\begin{equation}
X_{m, \beta} = \{x \in X_m : \overline{x^{m, \beta}} \preceq \sigma^k x \preceq x^{m, \beta}, \forall k \in \mathbb{N} \cup \{0\}\} \,.
\label{symmetric-beta-shift}
\end{equation}

A more detailed analysis of these dynamical systems can be found in \cite{MR3223814, MR3570134}. Through this paper we are going to use the following metric on $X_{m, \beta}$
\[
d(x, y) = 2^{-\min\{n \in\mathbb{N} : x(n) \neq y(n)\} + 1} \,. 
\]

It is easy to check that for any $m \in \mathbb{N}$ and each $\beta \in (\mathcal{G}(m), m + 1]$ the metric space $(X_{m, \beta}, d)$ is a bounded metric space. 

 Moreover, in \cite{MR3570134} it was proved that $X_{m, \beta}$ is a compact $\sigma$-invariant subset of $X_m$. 

From now on, we will denote by $\sigma = \sigma|_{X_{m, \beta}}$. Observe that $(X_{m, \beta}, \sigma)$ is a compact subshift.

 In \cite{MR3896110} it was defined the set of irreducible sequences, that is, the set of sequences $(x(n))_{n \in \mathbb{N}} \in X_{m, \beta}$ satisfying 
\begin{equation}
x(1) \ldots x(j) (\overline{x(1) \ldots x(j)}^+)^{\infty} \prec (x(n))_{n \in \mathbb{N}}, \forall j \in \mathbb{N} \,.
\label{irreducible-sequence}
\end{equation}

Besides that, we define $\beta_T$ as the unique number belonging to the interval $(1, m + 1]$ satisfying that the quasi-greedy $\beta_T$-expansion of $1$ is given by 
\begin{equation}
x^{m, \beta_T} :=
\begin{cases}
(k + 1)k^{\infty}, & m = 2k \\
(k + 1)((k + 1)k)^{\infty}, & m = 2k + 1  
\end{cases} \nonumber
\label{transitive-base}
\end{equation}

One of the main results in \cite{MR3896110} claims that for any $\beta \in (\mathcal{G}(m), m + 1] \cap \overline{\mathcal{U}_m}$, the symmetric $\beta$-shift $X_{m, \beta}$ is a topologically transitive subshift, if and only if, the quasi-greedy $\beta$-expansion of 1 is an irreducible sequence, or $\beta = \beta_T$.

The so called  {\bf transitivity condition} will be necessary to guarantee existence of an strictly positive eigenfunction and an eigenmeasure associated to the Ruelle operator. Henceforth, we are going to assume that either $\beta \in (\mathcal{G}(m), m + 1] \cap \overline{\mathcal{U}_m}$, with $x^{m, \beta}$ an irreducible  sequence, or, $\beta = \beta_T$. 

We denote by $\mathcal{M}_{\sigma}(X_{m, \beta})$ the set of {\bf $\sigma$-invariant probabilities} for the shift acting on $ X_{m, \beta}$, that is, the set of probability measures satisfying $\mu(E) = \mu(\sigma{-1}(E))$ for any Borelian set $E \subset X_{m, \beta}$.

Our first result provides conditions to guarantee the existence of Gibbs states for H\"older potentials defined on the symmetric $\beta$-shift $X_{m, \beta}$ associated to certain values  of $m \in \mathbb{N}$ and $\beta \in (\mathcal{G}(m), m + 1]$. The statement of the main result is the following:

\begin{theorem}
Consider $X_{m, \beta}$ a symmetric $\beta$-shift satisfying the transitivity condition. Let $A : X_{m, \beta} \to \mathbb{R}$ be a H\"older continuous potential. There exists a class of possible values of $m$ and $\beta$, such that, there exists $\lambda_A > 0$ and $\psi_A : X_{m, \beta} \to \mathbb{R}$, a strictly positive H\"older continuous function, such that, $\mathcal{L}_A(\psi_A) = \lambda_A\psi_A$.

The eigenvalue $\lambda_A$ is simple and is the maximal possible eigenvalue. Moreover, there exists a unique Radon probability measure $\rho_A$, defined on the Borelian sets of $X_{m, \beta}$, such that, $\mathcal{L}^*_A(\rho_A) = \lambda_A\rho_A$. 

The invariant probability measure $\mu_A = \psi_A d\rho_A$ is the unique fixed point of $\mathcal{L}^*_{\overline{A}}$, where $\overline{A}$ is the normalization of $A$. Furthermore, for any H\"older continuous potential $\psi : X_{m, \beta} \to \mathbb{R}$, it is satisfied the following uniform limit
\[
\lim_{n \in \mathbb{N}}\lambda^{-n}_A\mathcal{L}^n_A(\psi) = \psi_A \int_{X_{m, \beta}} \psi d\rho_A \,.
\]

We will present  explicit conditions  for values of $m$ and $\beta$ such that all the claims are valid. For instance, for $m > 2$ and $\beta \in [\frac{m}{2} + 2, m + 1]$.

\label{theorem1}
\end{theorem} 

Theorem \ref{theorem1} assures that Ruelle's Perron-Frobenius Theorem can be proved for some parameters $m$ and $\beta$ of the symmetric $\beta$-shift, which implies that given a H\"older continuous function $A$,  the value $\lambda_A$, the eigenfunction $\psi_A$ and te eigen-probability  $\rho_A$ are well defined.  The determination of these parameters is the main result of this paper. The proofs of the  other results we presented here are in some way analogous to the proof of other known results (but there are some technical differences).

Remember that  $\mathcal{M}_{\sigma}(X_{m, \beta})$ denotes the set of invariant probabilities for the shift acting on $ X_{m, \beta}.$

We define a suitable definition of entropy on section \ref{pre} and we show a variational principle of pressure.
In Proposition \ref{variational-principle} we will show that the probability $\psi_A\, \rho_A$
maximizes pressure.

\medskip

Hereafter, we will use the notation $m(A) = \sup_{\mu \in \mathcal{M}_{\sigma}(X_{m, \beta})} \left\{\int_{X_{m, \beta}} A d\mu \right\}$,  and we will denote by $\mathcal{M}_{\max}(A)$ the set of {\bf $A$-maximizing probability measures}, that is, the set of $\sigma$-invariant probability measures that attain $m(A)$. It is easy to check that $\mathcal{M}_{\max}(A)$ is a compact non-empty set.

Theorem \ref{theorem1} and the variational principle described in the  Proposition \ref{variational-principle} will ensure the existence of a family of equilibrium states $(\mu_{t\,A})_{t>0}$ depending of the real parameter $t$ (and a suitable expression for the family of entropies $(h(\mu_{t\,A}))_{t>0}$). 

Moreover, it follows from the compactness of set of shift-invariant probabilities on $X_{m, \beta}$ that the family $(\mu_{t\,A})_{t>0}$ has accumulation points, when $t \to\infty$. 

These accumulation points are sometimes called ground states. The parameter $t$ is usually identified with the inverse of the temperature for  the system of particles on the lattice under the action of the potential $A$. The limit of $\mu_{t\,A}$, when $t\to \infty$, is called the {\bf limit at zero temperature}.   It is known that the ground state probabilities are maximizing probabilities for the potential $A$ (see \cite{BLL}, \cite{MR3377291} or \cite{Lele}).

One interesting question is what can be said about the entropy of ground states  in terms of the entropies of equilibrium states, when temperature goes to zero, that is, the values $(h(\mu_{t\,A}))_{t>0}$, when $t \to \infty$ . 

One of the first works in this direction was due to Contreras et al. in \cite{MR1855838}. They show some properties of the limit of this family of entropies at zero temperature for potentials of class $\mathcal{C}^{1 + \alpha}$ defined on $S^1$. In the non-compact case, Morris proved in \cite{MR2295238} the existence of the zero temperature limit (see also \cite{MR2151222}). All these results were extended recently in \cite{MR3864383} by Freire and Vargas beyond the finitely primitive case. 

Although these type of problems have been widely studied in finite type subshifts in both, compact the non-compact setting, they have not been studied in a non-Markovian setting yet. Our second result guarantees the existence of the zero temperature limit for entropies in the   symmetric $\beta$-shifts model. 
The statement of the result is as follows:    

\begin{theorem} 
Consider $m > 2$ and $\beta \in [\frac{m}{2} + 2, m + 1]$, such that,  $X_{m, \beta}$ satisfies the transitivity condition. Let $A : X_{m, \beta} \to  \mathbb{R}$ be a H\"older continuous potential. Then, the family $(h(\mu_{tA}))_{t > 0}$ is continuous at infinity, and 
\[
\lim_{t \to \infty} h(\mu_{tA}) = \max_{\mu \in \mathcal{M}_{max}(A)} h(\mu) \,.
\]
\label{theorem3}
\end{theorem}

The paper is organized as follows.

In section \ref{preliminaries-section} we present some preliminaries and  we introduce the Ruelle operator on symmetric $\beta$-shifts. We prove that it is well defined and preserves the set of H\"older continuous functions. At the end of the section we introduce some more  notation that will be used through the paper. 

In the Appendix \ref{RPF-theorem-section} appears the proof of the Theorem \ref{theorem1} (which is similar to the one in \cite{MR1085356}).
 
 In section \ref{pre} we  present a suitable definition of entropy.  We also  consider a variational principle of pressure. We use the results  above to show that the Gibbs state found in the Ruelle-Perron-Frobenius Theorem is an equilibrium state.
 
  Finally, in section \ref{zero-temperature-limit-section} we present the proof of Theorem \ref{theorem3}.

\section{The Ruelle operator - existence of eigenfunctions and eigenprobabilities}
\label{preliminaries-section}

In this section we present the definition of the Ruelle operator and we show that it is well defined for certain values of the parameters $m$ and $\beta$. 

By the characterization of symmetric $\beta$-shifts that appears in (\ref{symmetric-beta-shift}) we get that for any pair of numbers $m \in \mathbb{N}$ and $\beta \in (1, m + 1]$, it is true that $X_{m, \beta} = X_{\mathcal{F}_{m, \beta}}$, with $\mathcal{F}_{m, \beta}$ the collection of forbidden words of the shift $X_{m, \beta}$ (See \cite{MR1369092}), given by 
\[
\mathcal{F}_{m, \beta} = \bigcup_{n \in \mathbb{N}} \bigl(\mathcal{F}_{m, \beta}(n) \cup \overline{\mathcal{F}_{m, \beta}}(n)\bigr) \,,
\]
with $\mathcal{F}_{m, \beta}(n) = \{\omega = \omega(1) \ldots \omega(n) : x^{m, \beta}(1) \ldots x^{m, \beta}(n) \prec \omega(1) \ldots \omega(n)\}$, and $\overline{\mathcal{F}_{m, \beta}}(n) = \{\omega = \omega(1) \ldots \omega(n) : \omega(1) \ldots \omega(n) \prec \overline{x^{m, \beta}(1) \ldots x^{m, \beta}(n)}\}$.

From the characterization of symmetric $\beta$-shifts in (\ref{symmetric-beta-shift}), we can define the cylinder associated to the word $\omega = \omega(1) \ldots \omega(l)$, as the set 
\[
[\omega] = \{x \in X_{m, \beta} : x(1) = \omega(1), \ldots, x(l) = \omega(l)\} \,.
\]

Observe that $[\omega] \neq \emptyset$, if and only if, $\overline{x^{m, \beta}(i)} \preceq \omega(i + k) \preceq x^{m, \beta}(i)$, for all $i \in \{1, \ldots, l\}$ and each $1 \leq k \leq l - i$. Moreover, the topology generated by cylinders coincides with the product topology on the set $X_{m, \beta}$ and $\mathcal{P} = \{[0], \ldots, [m]\}$ is a generating partition of the Borel sigma algebra. 

It is easy to check that $X_{m, \beta}$ is a completely disconnected set. Indeed, if $U$ is a non-empty connected open set satisfying $U \subset X_{m, \beta}$ and $U \neq \{x\}$, we can choose $x \in U$ and $\epsilon > 0$, such that, $B(x, \epsilon) \subset U$. Therefore, for each $y \in B(x, \epsilon)$, it is satisfied $\overline{x^{m, \beta}} \preceq \sigma^k y \preceq x^{m, \beta}$, for all $k \in \mathbb{N} \cup \{0\}$, in other words, $\sigma^k(B(x, \epsilon)) \subset X_{m, \beta}$, for all $k \in \mathbb{N} \cup \{0\}$, which is a contradiction. This is so, because for each $y_0 \in B(x, \epsilon)) \setminus \{x\}$, there exists $k_0 \in \mathbb{N}$, such that, $d(\sigma^{k_0}(x), \sigma^{k_0}(y_0)) = 2^{k_0}\epsilon > d(\overline{x^{m, \beta}}, x^{m, \beta})$.

We will be interested only in the case of symmetric $\beta$-shifts $X_{m, \beta}$, with $m \in \mathbb{N}$ and values of $\beta \in (\mathcal{G}(m), m + 1] \cap \overline{\mathcal{U}_m}$.

Hereafter, we will denote by $\mathcal{H}_{\alpha}(X_{m, \beta})$ the set of {\bf H\"older continuous functions} from $X_{m, \beta}$ into $\mathbb{R}$ with coefficient $\alpha$, i.e. the set of functions $\psi : X_{m, \beta} \to \mathbb{R}$ satisfying for some $K \geq 0$ and all $x, y \in X_{m, \beta}$ the following inequality
\begin{equation}
|\psi(x) - \psi(y)| \leq Kd(x, y)^{\alpha} \,.
\label{Holder}
\end{equation} 

Besides that, for any $\psi \in \mathcal{H}_{\alpha}(X_{m, \beta})$ we will use the notation $\mathrm{Hol}_{\psi}$ for the H\"older constant of $\psi$, which is defined by $\mathrm{Hol}_{\psi} = \sup_{x \neq y} \frac{|\psi(x) - \psi(y)|}{d(x, y)^{\alpha}}$. Thus, given $\psi \in \mathcal{H}_{\alpha}(X_{m, \beta})$, its norm is defined as $\|\psi\|_{\alpha} = \|\psi\|_{\infty} + \mathrm{Hol}_{\psi}$. It is simple to check that $(\mathcal{H}_{\alpha}(X_{m, \beta}), \|\cdot\|)$ is a Banach space. 
 
We will denote by $\mathcal{C}(X_{m, \beta})$ the set of continuous functions from $X_{m, \beta}$ into $\mathbb{R}$. Taking a potential $A \in \mathcal{C}(X_{m, \beta})$, we define the {\bf Ruelle operator associated to $A$} as the function that assigns to each continuous function $\varphi$, the function
\begin{equation}
\mathcal{L}_A(\varphi)(x) := \sum_{\sigma(y) = x} e^{A(y)}\varphi(y) \,.
\label{Ruelle-operator}
\end{equation}

In the following Lemma we will provide values for $m$ and $\beta$, such that, $\mathcal{L}_{A}(\varphi)$ is well defined on the set $X_{m, \beta}$. 

\begin{lemma} \label{bigd}
Consider $m > 2$ and $\beta \in \left[\frac{m}{2} + 2, m + 1\right]$. Then, any $x \in X_{m, \beta}$ is such  that $\sigma^{-1}(\{x\}) \neq \emptyset$. Moreover, in this case $\#(\sigma^{-1}(\{x\})) \geq 2$.
\end{lemma}
\begin{proof}
We use the notation $ax$ for the sequence $(a, x(1), x(2), \ldots) \in \mathcal{A}_m^{\mathbb{N}}$. From the above, we define for each $x \in X_{m, \beta}$ the set $\mathcal{A}_m(x)$ as 
\begin{align}
\mathcal{A}_m(x) 
&:= \{a \in \mathcal{A}_m: ax \in X_{m, \beta}\} \nonumber \\
&= \{a \in \mathcal{A}_m: ax(1) \ldots x(n) \notin \mathcal{F}_{m, \beta}, \, \forall n \in \mathbb{N} \} \nonumber \,.
\end{align}

Moreover, each $a \in \mathcal{A}_m(x)$ satisfies the following conditions:
\begin{enumerate}
\item $\frac{m}{\beta - 1} - 1 < \frac{a}{\beta} < 1$;
\item $\frac{m}{\beta - 1} - 1 < \frac{a}{\beta} + \frac{1}{\beta}\sum_{k=1}^n x(k)\beta^{-k} < 1$ for each $n \in \mathbb{N}$.
\end{enumerate}

Note that by the above definition, it follows immediately that 
\[
\#(\sigma^{-1}(\{x\})) = \#(\mathcal{A}_m(x)) \,.
\]

Besides that, each point $x \in X_{m, \beta}$ satisfy the following inequalities 
\begin{equation}
\frac{m}{\beta - 1} - 1 < \sum_{k = 1}^n x(k)\beta^{-k} < 1 , \, \forall n \in \mathbb{N} \,.
\label{inequality-shift}
\end{equation}

Now, we want to demonstrate that under the hypothesis of this Lemma, it is satisfied 
\begin{equation}
\left(\frac{\beta}{\beta - 1}(m - \beta + 1), \beta - 1\right) \cap \mathbb{N} \subset \mathcal{A}_m(x) \,. 
\label{interval}
\end{equation}

Indeed, taking 
\[
a \in \left(\frac{\beta}{\beta - 1}(m - \beta + 1), \beta - 1\right) \cap \mathbb{N} \,,
\]

by the right side in (\ref{inequality-shift}), we obtain that for all $n \in \mathbb{N}$ it is satisfied
\[
\frac{a}{\beta} + \frac{1}{\beta}\sum_{k = 1}^n x(k)\beta^{-k} < \frac{a}{\beta} + \frac{1}{\beta} < \frac{\beta - 1}{\beta} + \frac{1}{\beta} = 1 \,
\]

On other hand, by the left side in (\ref{inequality-shift}), for all $n \in \mathbb{N}$, we have 
\begin{align}
\frac{a}{\beta} + \frac{1}{\beta}\sum_{k = 1}^n x(k)\beta^{-k} 
&> \frac{m - \beta + 1}{\beta - 1} + \frac{m}{\beta(\beta - 1)} - \frac{1}{\beta} \nonumber \\
&> \frac{m - \beta + 1}{\beta} + \frac{m}{\beta(\beta - 1)} - \frac{1}{\beta} \nonumber \\ 
&= \frac{(m - \beta + 1)(\beta - 1) + m - (\beta - 1)}{\beta(\beta - 1)} \nonumber \\
&= \frac{m\beta - \beta^2 + \beta - m + \beta -1 + m - \beta + 1}{\beta(\beta - 1)} \nonumber \\
&= \frac{m - \beta + 1}{\beta - 1} \nonumber \\
&=\frac{m}{\beta - 1} - 1 \nonumber \,.
\end{align}

Therefore, 
\[
\frac{m}{\beta - 1} - 1 < \frac{a}{\beta} + \frac{1}{\beta}\sum_{k=1}^n x(k)\beta^{-k} < 1, \, \forall n \in \mathbb{N} \,.
\]

Besides that, as $a < \beta - 1$, we get
\[
\frac{a}{\beta} < \frac{\beta - 1}{\beta} < 1 \,,
\]

and using the fact that $a > \frac{\beta}{\beta - 1}(m - \beta + 1)$, it follows that
\[
\frac{a}{\beta} > \frac{m - \beta + 1}{\beta - 1} = \frac{m}{\beta - 1} - 1 \,.  
\]

That is, 
\[
\frac{m}{\beta - 1} - 1 < \frac{a}{\beta} < 1 \,.
\]

By the above and (\ref{preimages}), it follows that $a \in \mathcal{A}_m(x)$, which proves (\ref{interval}). 

Moreover, taking $\beta = \frac{m}{2} + 2$, it follows that
\[
\left(\frac{\beta}{\beta - 1}(m - \beta + 1), \beta - 1\right) = \left(\frac{m^2/4 + m/2 - 2}{m/2 + 1}, \frac{m}{2} + 1\right) \,,
\]

and for all $m > 2$ it is satisfied
\[
\#\bigl(\left(\frac{m^2/4 + m/2 - 2}{m/2 + 1}, \frac{m}{2} + 1\right) \cap \mathbb{N} \bigr) \geq 2\,.
\]

In addition, we have $\left(\frac{m^2/4 + m/2 - 2}{m/2 + 1}, \frac{m}{2} + 1\right) \subseteq \left(\frac{\beta}{\beta - 1}(m - \beta + 1), \beta - 1\right)$, for all $\beta \in \left[\frac{m}{2} + 2, m + 1\right]$. 

Therefore, it follows that $\#(\sigma^{-1}(\{x\})) = \#(\mathcal{A}_m(x)) \geq 2$, for all $\beta \in \left[\frac{m}{2} + 2, m + 1\right]$
\end{proof}

In the next Lemma we will check that the Ruelle operator is a local homeomorphism. Moreover, we will show in the next Lemma that the Ruelle operator preserves the set of 
H\"older continuous potentials.

\begin{lemma}
Consider $x \in X_{m, \beta}$ such that $\mathcal{A}_m(x) \neq \emptyset$. Then, for any point $x' \in X_{m, \beta}$ which is close enough to $x$, we get $\mathcal{A}_m(x) = \mathcal{A}_m(x')$.
\end{lemma}

\begin{proof}
For a fixed $x \in X_{m, \beta}$ and $a \in \mathcal{A}_m(x)$, it is easy to verify that 
\[
\frac{\beta m}{\beta - 1} - \beta - a \leq \sum_{n \in \mathbb{N}}x(n)\beta^{-n} \leq \beta - a \,.
\]

The analysis of the above inequalities  can be decomposed in the analysis of the following cases:
 \medskip
 
\textit{Case 1:} 
\[
\frac{\beta m}{\beta - 1} - \beta - a < \sum_{n \in \mathbb{N}}x(n)\beta^{-n} < \beta - a \,.
\]

Take  $N \in \mathbb{N}$ large enough, such that,
\[
m \sum_{n > N} \beta^{-n} < \min \left\{\beta - a - \sum_{n \in \mathbb{N}}x(n)\beta^{-n}, \sum_{n \in \mathbb{N}}x(n)\beta^{-n} - \frac{\beta m}{\beta - 1} + \beta + a \right\} \,.
\] 

It follows that for any $x' \in X_{m, \beta}$ with $d(x, x') < 2^{-N}$, it is satisfied
\[
\frac{\beta m}{\beta - 1} - \beta - a < -m \sum_{n > N} \beta^{-n} + \sum_{n \in \mathbb{N}}x(n)\beta^{-n} \leq \sum_{n \in \mathbb{N}}x'(n)\beta^{-n} \,,
\]

and 
\[
\sum_{n \in \mathbb{N}}x'(n)\beta^{-n} \leq \sum_{n \in \mathbb{N}}x(n)\beta^{-n} + m \sum_{n > N} \beta^{-n} < \beta - a \,.
\]

The reasoning above implies $ax' \in X_{m, \beta}$, which is equivalent to $a \in \mathcal{A}_m(x')$.

\medskip

\textit{Case 2:}
\[
\sum_{n \in \mathbb{N}}x(n)\beta^{-n} = \beta - a \,.
\]

For this case we choose $N \in \mathbb{N}$ large enough, such that 
\[
m \sum_{n > N} \beta^{-n} < \sum_{n \in \mathbb{N}}x(n)\beta^{-n} - \frac{\beta m}{\beta - 1} + \beta + a \,.
\] 

Therefore, for any $x' \in X_{m, \beta}$ satisfying $x' \preceq x$ and $d(x, x') < 2^{-N}$, we have
\[
\frac{\beta m}{\beta - 1} - \beta - a < -m \sum_{n > N} \beta^{-n} + \sum_{n \in \mathbb{N}}x(n)\beta^{-n} \leq \sum_{n \in \mathbb{N}}x'(n)\beta^{-n} \,,
\]

and 
\[
\sum_{n \in \mathbb{N}}x'(n)\beta^{-n} \leq \sum_{n \in \mathbb{N}}x(n)\beta^{-n} = \beta - a \,.
\]

By the above, $ax' \in X_{m, \beta}$. Then, we can conclude that $a \in \mathcal{A}_m(x')$.

\medskip

\textit{Case 3:} 
\[
\frac{\beta m}{\beta - 1} - \beta - a = \sum_{n \in \mathbb{N}}x(n)\beta^{-n} \,.
\]

In this case, we choose $N \in \mathbb{N}$ large enough, such that,
\[
m \sum_{n > N} \beta^{-n} < \beta - a - \sum_{n \in \mathbb{N}}x(n)\beta^{-n} \,.
\] 

Thus, for any $x' \in X_{m, \beta}$ satisfying $x \preceq x'$ and $d(x, x') < 2^{-N}$, we get that 
\[
\frac{\beta m}{\beta - 1} - \beta - a = \sum_{n \in \mathbb{N}}x(n)\beta^{-n} \leq \sum_{n \in \mathbb{N}}x'(n)\beta^{-n} \,,
\]

and 
\[
\sum_{n \in \mathbb{N}}x'(n)\beta^{-n} \leq \sum_{n \in \mathbb{N}}x(n)\beta^{-n} + m \sum_{n > N} \beta^{-n} < \beta - a \,.
\]

By the above, $ax' \in X_{m, \beta}$. That is, $a \in \mathcal{A}_m(x')$.

Note that in all the cases studied above  $\mathcal{A}_m(x) = \mathcal{A}_m(x')$, when $x$ and $x'$ are close enough. The foregoing implies that $\mathcal{L}_A(\varphi)$ is a local homeomorphism when $A, \varphi \in \mathcal{C}(X_{m, \beta})$.
\end{proof}

\medskip

There are parameters $m$ and $\beta$  (for instance, when $m > 2$ and $\beta \in [\frac{m}{2} + 2, m + 1]$), such that, the shift is transitive and also satisfy the conditions of Lemma \ref{bigd}. We assume on the proof of the Ruelle Theorem these conditions.

\medskip

The main conclusion is: if $A, \varphi \in \mathcal{H}_{\alpha}(X_{m, \beta})$ and $x, x' \in X_{m, \beta}$ are close enough, we get that $\mathcal{A}_m(x) = \mathcal{A}_m(x')$. It follows that
\begin{align}
|\mathcal{L}_{A}(\varphi)(x) - \mathcal{L}_{A}(\varphi)(x')|
&\leq \sum_{a \in \mathcal{A}_m(x)}\left|e^{A(ax)}\varphi(ax) - e^{A(ax')}\varphi(ax')\right| \nonumber \\
&\leq \frac{(m + 1)}{2^{\alpha}}\bigl(e^{\|A\|_{\infty}}\mathrm{Hol}_{\varphi} + \|\varphi\|_{\infty}\mathrm{Hol}_{e^{A}}\bigr)d(x, x')^{\alpha} \nonumber \,.
\end{align}

By the above, $\mathcal{L}_{A}(\varphi)$ is locally H\"older continuous. Thus, by compactness of $X_{m, \beta}$, it follows that $\mathcal{L}_{A}(\varphi) \in \mathcal{H}_{\alpha}(X_{m, \beta})$.

It follows that for each $n \in \mathbb{N}$, the $n$-th iterate of the Ruelle operator, which is defined by 
\[
\mathcal{L}^n_A(\varphi)(x) = \sum_{\sigma^n(y) = x} e^{S_n A(y)}\varphi(y) \,
\] 
satisfies the same properties mentioned above, where $S_n A(y) = \sum_{j=0}^{n-1} A(\sigma^j(y))$. The above, it is an immediate consequence of the fact $\mathcal{L}^n_A(\varphi) = \mathcal{L}_A\left(\mathcal{L}^{n-1}_A(\varphi)\right)$

Given two Banach spaces $X$ and $Y$ we denote by $l(X, Y)$ the Banach space of linear continuous operators from $X$ into $Y$. We will use the notation $l(X)$ for the Banach space of linear continuous operators from $X$ into itself.

Note that $\|\mathcal{L}_A(\varphi)\|_{\alpha} < K < \infty$, for any $\varphi \in \mathcal{H}_{\alpha}(X_{m, \beta})$, with $\|\varphi\|_{\alpha} \leq 1$. Therefore, $\|\mathcal{L}_A\| < \infty$, in other words $\mathcal{L}_A \in l(\mathcal{H}_{\alpha}(X_{m, \beta}))$. 

Using the properties of the dual space of a Banach space, we can define the dual of the Ruelle operator $\mathcal{L}^*_A$ on the set of Radon measures, as the operator that satisfies for any $\varphi \in \mathcal{C}(X_{m, \beta})$ the following equation 
\[
\int_{X_{m, \beta}}\varphi d\left(\mathcal{L}^*_A(\nu)\right) = \int_{X_{m, \beta}}\mathcal{L}_A(\varphi) d\nu \,.
\]

From the above equation, for each $n \in \mathbb{N}$ we can express the $n$-th iterate of the dual Ruelle operator by
\[
\int_{X_{m, \beta}}\varphi \,\,d\left(\mathcal{L}^{*,n}_A(\nu)\right) = \int_{X_{m, \beta}}\mathcal{L}^n_A(\varphi) d\nu \,.
\]

From now on, we will denote by $\mathcal{R}(X_{m, \beta})$ the set of Radon measures on the symmetric $\beta$-shift $X_{m, \beta}$ and we will use the notation $\mathcal{M}(X_{m, \beta})$ for the set of Radon probability measures on $X_{m, \beta}$. Besides that, we will denote by $\mathcal{M}_{\sigma}(X_{m, \beta})$ the set of $\sigma$-invariant Radon probability measures defined on $X_{m, \beta}$. Observe that by Banach-Alaoglu's Theorem both of the sets, $\mathcal{M}(X_{m, \beta})$ and $\mathcal{M}_{\sigma}(X_{m, \beta})$, are compact subsets of $\mathcal{R}(X_{m, \beta})$.

Theorem \ref{theorem1}  claims that if $A : X_{m, \beta} \to \mathbb{R}$ is a  H\"older continuous potential, then, there exists $\lambda_A > 0$, 
and

1) a  unique Radon probability measure $\rho_A$, defined on the Borelian sets of $X_{m, \beta}$, such that, $\mathcal{L}^*_A(\rho_A) = \lambda_A\rho_A$. 

2) a  function $\psi_A : X_{m, \beta} \to \mathbb{R}$ which is a  strictly positive H\"older continuous function and such that $\mathcal{L}_A(\psi_A) = \lambda_A\psi_A$. Assuming that $\int \psi_A d \rho_a =1$, the normalized eigenfunction $\psi_A$ is uniquely determined (because the probability $\rho_A$ was {\bf uniquely} determined).

3) The probability measure $\mu_A = \psi_A d\rho_A$, where we take the uniquely determined function $\psi_A$ from item 2), is invariant for the shift and also uniquely determined.

Hereafter, we will assume that  $\rho_A$, $\psi_A $ and $\mu_A$ denote the uniquely determined elements describe by 1) , 2) and 3) (and in this order of determination).

In the proof of Theorem \ref{theorem1} we use a similar procedure as the one that appears in \cite{MR1085356} for the case of compact  subshifts. One can  show that the same reasoning can be applied on our setting (we have to check all details for a specific proof that works in our setting). For a question of completeness we will present the sketch of the proof in the appendix \ref{RPF-theorem-section}.

\section{The variational principle of pressure} \label{pre}

In this section we are going to define the entropy associated to a $\sigma$-invariant probability measure. Furthermore, we are going to show that this definition satisfies a variational principle. Indeed, given $\mu \in \mathcal{M}_{\sigma}(X_{m, \beta})$ we define the entropy of $\mu$ (see \cite{MR3377291}) as 
\begin{equation}
h(\mu) = \inf_{u \in \mathcal{C}^+(X_{m, \beta})}\left\{\int_{X_{m, \beta}}\log\left(\frac{\mathcal{L}_0(u)}{u}\right)d\mu\right\} \,.
\label{entropy}
\end{equation}

We assume that the parameters $m$ and $\beta$ satisfy the conditions required in last section. That is, $m > 2$ and $\beta \in [m/2 + 2, m + 1]$.

Given the potential  $A \in \mathcal{H}_{\alpha}(X_{m, \beta})$, the normalization of $A$ is defined as 
\begin{equation}
\overline{A} := A + \log(\psi_A) - \log(\psi_A \circ \sigma) - \log(\lambda_A) \,.
\label{normalization}
\end{equation}

The above definition will be used in the proof of the variational principle that appears in Proposition \ref{variational-principle} and also in the proof of Theorem \ref{theorem1} that appears in section \ref{RPF-theorem-section}.

If $\mu$ is a fixed point of the Ruelle operator associated to some H\"older continuous potential (see \cite{MR3377291}) the following Lemma show us that entropy (given by the above definition)
can be expressed in an integral form.

\begin{lemma}
Consider $A \in \mathcal{H}_{\alpha}(X_{m, \beta})$ and $\mu_A$ the unique fixed point of the dual Ruelle operator $\mathcal{L}^*_{\overline{A}}$. Then, 
\[
h(\mu_A) = -\int_{X_{m, \beta}} \overline{A}d\mu_A \,. 
\]
\label{variational-principle-entropy}
\end{lemma}
\begin{proof}
Set $u_0 = e^{\overline{A}}$, so $u_0 \in \mathcal{C}^+(X_{m \beta})$. Since $\mathcal{L}^*_A(\mu_A) = \mu_A$, it follows that
\[
-\int_{X_{m, \beta}}\overline{A} d\mu_A 
= \int_{X_{m, \beta}}\log\left(\frac{\mathcal{L}_{\overline{A}}(1)}{u_0}\right) d\mu_A 
= \int_{X_{m, \beta}}\log\left(\frac{\mathcal{L}_0(u_0)}{u_0}\right) d\mu_A \,.
\]

On other hand, for any $\widetilde{u} \in \mathcal{C}^+(X_{m \beta})$, we have that $u = \widetilde{u}e^{-\overline{A}} \in \mathcal{C}^+(X_{m \beta})$ and $\mathcal{L}_0(\widetilde{u}) = \mathcal{L}_{\overline{A}}(u)$. Therefore,
\[
\log\left(\frac{\mathcal{L}_0(\widetilde{u})}{\widetilde{u}}\right) = \log\left(\frac{\mathcal{L}_{\overline{A}}(u)}{u}\right) - \overline{A} \,.
\]

From the above, integrating with respect to $\mu_A$, we get that
\[
\int_{X_{m, \beta}}\log\left(\frac{\mathcal{L}_0(\widetilde{u})}{\widetilde{u}}\right)d\mu_A 
=\int_{X_{m, \beta}}\log\left(\frac{\mathcal{L}_{\overline{A}}(u)}{u}\right)d\mu_A  - \int_{X_{m, \beta}} \overline{A} d\mu_A \,.
\]

Since $\overline{A}$ is a normalized potential, from the Jensen's inequality, it follows that $0 \geq \int_{X_{m, \beta}}\log(\mathcal{L}_{\overline{A}}(u))d\mu_A - \int_{X_{m, \beta}}\log(u)d\mu_A$. Therefore,  
\[
\int_{X_{m, \beta}}\log\left(\frac{\mathcal{L}_0(\widetilde{u})}{\widetilde{u}}\right)d\mu_A \geq - \int_{X_{m, \beta}} \overline{A} d\mu_A \,.
\]

In other words, we have 
\[
- \int_{X_{m, \beta}} \overline{A} d\mu_A = \inf_{u \in \mathcal{C}^+(X_{m, \beta})}\left\{\int_{X_{m, \beta}}\log\left(\frac{\mathcal{L}_0(u)}{u}\right)d\mu_A\right\} = h(\mu_A) \,.
\]

\end{proof}

The next Proposition shows that the Gibbs state found in Theorem \ref{theorem1} satisfies the variational principle. Note that the above implies that any Gibbs state $\mu_A$ associated to some H\"older continuous potential $A$ (defined on the symmetric $\beta$-shift) is in fact an equilibrium state. 

\begin{proposition}
Given $A \in \mathcal{H}_{\alpha}(X)$ the topological pressure $P(A)$ of the potential $A$ is defined by
\[
P(A) = \sup_{\mu \in \mathcal{M}_{\sigma}(X_{m, \beta})}\left\{ h(\mu) + \int_{X_{m, \beta}}A d\mu \right\} \,.
\]

Then, $P(A)= \log(\lambda_A) $, where $\lambda_A$ is the eigenvalue of the Ruelle operator.

The probability which attains the maximal value is $\mu_{\overline{A}} $, where 
$\overline{A}$
is associated to $A$ via the expression
(\ref{normalization}). It is also true that $\mu_{\overline{A}}=\mu_A = \psi_A d\rho_A$ (see 3) in the end of section 2).

\label{variational-principle}
\end{proposition}
\begin{proof}
By Lemma \ref{variational-principle-entropy} we have 
\[
P(A) = \log(\lambda_A) = h(\mu_A) + \int_{X_{m, \beta}}A d\mu_A \,.\]

Besides that, for any $\mu \in \mathcal{M}_{\sigma}(X_{m, \beta})$, it is satisfied
\begin{align}
h(\mu) 
&= \inf_{u \in \mathcal{C}^+(X_{m, \beta})}\left\{\int_{X_{m, \beta}}\log\left(\frac{\mathcal{L}_0(u)}{u}\right)d\mu\right\} \nonumber \\
&\leq \int_{X_{m, \beta}}\log\left(\frac{\mathcal{L}_0(e^{\overline{A}})}{e^{\overline{A}}}\right)d\mu \nonumber \\
&= -\int_{X_{m, \beta}} \overline{A} d\mu \nonumber \\
&= -\int_{X_{m, \beta}} A d\mu + \log(\lambda_A) \nonumber \,.
\end{align}

\end{proof}

\section{Zero Temperature Limits for Entropies}
\label{zero-temperature-limit-section}

In section \ref{pre} we presented a notion of entropy for $\sigma$-invariant probability measures defined on symmetric $\beta$-shifts. Besides that, we showed in last section  that it is satisfied a variational principle and the supremum of the variational equation of pressure is attained at the Gibbs state associated to the potential $A$, which assures that any Gibbs state it is an equilibrium state as well. In this section, we are going to present the proof of Theorem \ref{theorem3} using the variational principle considered in Proposition \ref{variational-principle}. This Theorem guarantees that the function assigning to each $t > 0$ the value $h(\mu_{tA})$ is continuous at infinity,  which is known as zero temperature limit for the entropies of the equilibrium states.

General results on maximizing probabilities can be found in \cite{MR1855838}, \cite{G1} and \cite{BLL}.

\begin{proof}[Proof of Theorem \ref{theorem3}] 
Note that any accumulation point of the family $(\mu_{tA})_{t > 0}$ in the weak* topology is an $A$-maximizing probability measure. Indeed, for any $\mu \in \mathcal{M}_{\sigma}(X_{m, \beta})$ and each $t > 0$, it is satisfied 
\begin{equation}
\frac{1}{t}h(\mu_{tA}) + \int_{X_{m, \beta}}A d\mu_{tA} \geq \frac{1}{t}h(\mu) + \int_{X_{m, \beta}}A d\mu \,. 
\label{asymptotic-pressure}
\end{equation}

Let $\mu_{\infty}$ be an accumulation point at $\infty$ of the family $(\mu_{tA})_{t > 0}$. Then, there exists an increasing sequence of positive real numbers $(t_n)_{n \in \mathbb{N}}$, such that, $\lim_{n \in \mathbb{N}} t_n = +\infty$ and $\lim_{n \in \mathbb{N}} \mu_{t_nA} = \mu_{\infty}$ in the weak* topology. Then, taking the limit, when $n \to \infty$, in (\ref{asymptotic-pressure}) and using the fact that $h(\mu_{t_nA}) \leq h(X_{m, \beta}) < \infty$, for all $n \in \mathbb{N}$, we get that 
\begin{align}
\int_{X_{m, \beta}}A d\mu_{\infty} 
&= \lim_{n \in \mathbb{N}}\left(\frac{1}{t_n}h(\mu_{t_nA}) + \int_{X_{m, \beta}}A d\mu_{t_nA}\right) \nonumber \\
&\geq \lim_{n \in \mathbb{N}}\left(\frac{1}{t_n}h(\mu) + \int_{X_{m, \beta}}A d\mu\right) \nonumber \\
&= \int_{X_{m, \beta}}A d\mu \nonumber \,.
\end{align}

The above implies that $m(A) = \int_{X_{m, \beta}}A d\mu_{\infty}$. On other hand, since for each cylinder $[i]$, with $i \in \mathcal{A}_m$, it is true that $\partial[i] = \emptyset$, then, the map $\mu \mapsto h(\mu)$ is upper semicontinuous. Thus, by compactness of $X_{m, \beta}$, it is guaranteed that $\mathcal{M}_{\max}(A)$ is a compact set as well. Therefore, there exists $\widehat{\mu} \in \mathcal{M}_{\sigma}(X_{m, \beta})$, such that, $h(\mu) \leq h(\widehat{\mu})$, for all $\mu \in \mathcal{M}_{\sigma}(X_{m, \beta})$. 
 
Note that (\ref{asymptotic-pressure}) implies that
\[
P(tA) = tm(A) + h(X_{m, \beta}) + o(t) \,. 
\]

That is, the topological pressure has an asymptote that depends of $m(A)$. This implies that 
\[
h(\widehat{\mu}) \leq h(X_{m, \beta}) + o(t) \,.
\] 

On other hand, by Proposition \ref{variational-principle}, we have 
\[
h(X_{m, \beta}) + o(t) = P(tA) - tm(A) \leq h(\mu_{tA}) \,.
\]

Therefore, 
\[
h(\widehat{\mu}) \leq \limsup_{t \to \infty} (h(X_{m, \beta}) + o(t)) \leq \limsup_{t \to \infty} h(\mu_{t A}) \leq h(\mu_{\infty}) \leq h(\widehat{\mu}) \,.
\]

Using again Proposition \ref{variational-principle}, we obtain that
\[
h(\mu_{tA}) \geq P(tA) - tm(A) \geq h(\widehat{\mu}) \,.
\]

The foregoing implies that for any $n \in \mathbb{N}$ it is satisfied the inequality
\[
\inf_{t \geq t_n} h(\mu_{t_n A}) \geq h(\widehat{\mu}) \,.
\]

Then, taking the limit when $n \to \infty$, we get
\[
\liminf_{t \to \infty} h(\mu_{tA}) \geq h(\widehat{\mu}) = \limsup_{t \to \infty} h(\mu_{t A}) \,. 
\]
\end{proof}

\section{Appendix -The proof of the Ruelle's Perron-Frobenius Theorem}
\label{RPF-theorem-section}

\begin{proof}[Proof of Theorem \ref{theorem1}]
In order to prove the Theorem \ref{theorem1} we will analyze first some properties of the following collection of functions 
\[
\Gamma := \{\psi \in \mathcal{C}(X_{m, \beta}) : 0 \leq \psi \leq 1, \log(\psi) \in \mathcal{H}_{\alpha}(X_{m, \beta})\} \,.
\]

\medskip

We will show first the existence of the eigenfunction. We assume that the parameters $m$ and $\beta$ are such that the action of the shift is transitive and the Ruelle operator is well defined.
The proof is quite similar to the one in \cite{MR1085356}. We just outline some of the steps.

\medskip

Note that the above $\Gamma$ is convex, because it is satisfied $\psi(x) \leq \psi(y)e^{\mathrm{Hol}_A d(x, y)^{\alpha}}$, for all $\psi \in \Gamma$. Moreover, the above inequality implies that
\begin{align}
|\psi(x) - \psi(y)| 
&\leq \|\psi\|_{\infty}\left(e^{\mathrm{Hol}_A d(x, y)^{\alpha}} - 1\right) \nonumber \\
&\leq \|\psi\|_{\infty} \mathrm{Hol}_A d(x, y)^{\alpha} e^{\mathrm{Hol}_A d(x, y)^{\alpha}} \nonumber \\
&\leq \|\psi\|_{\infty} \mathrm{Hol}_A e^{\mathrm{Hol}_A}d(x, y)^{\alpha} \nonumber \,.
\end{align} 

Therefore, $\Gamma$ is contained in $\mathcal{H}_{\alpha}(X_{m, \beta})$ and the same inequality implies that $\Gamma$ is an equicontinuous and uniformly bounded collection of functions, which implies that $\Gamma$ is uniformly compact by Arzela-Ascoli's Theorem.

Now, for each $k \in \mathbb{N}$ we define the operator $L_k$ from $\Gamma$ into $\Gamma$ by the following expression 
\[
L_k(\psi) = \frac{\mathcal{L}_A(\psi + 1/k)}{\|\mathcal{L}_A(\psi + 1/k)\|_{\infty}}\,.
\]

Since all the constant functions taking values in $[0, 1]$ belongs to the set $\Gamma$, by linearity of the Ruelle operator, we get that $L_k(\psi) \in \Gamma$, for any $\psi \in \Gamma$. Besides that, $\|L_k(\psi)\|_{\infty} = 1$, for all $k \in \mathbb{N}$. Thus, using convexity and uniformly compactness of $\Gamma$, it is guaranteed the existence of a fixed point $\psi_k$ for $L_k$ by Schauder-Tychonoff's Theorem. So,
\[
\mathcal{L}_A(\psi_k + 1/k) = \psi_k \|\mathcal{L}_A(\psi_k + 1/k)\|_{\infty} \,.
\]

Using again that $\Gamma$ is uniformly compact, we obtain that the sequence $(\psi_k)_{k \in \mathbb{N}}$ has an accumulation point $\psi_A$ with the uniform norm. Then, by continuity of $\mathcal{L}_A$ we have
\[
\mathcal{L}_A(\psi_A) = \psi_A \|\mathcal{L}_A(\psi_A)\|_{\infty} \,.
\] 

Hereafter, the last term in the right side of the equation will be denoted by $\lambda_A$. Observe that $\lambda_k = \|\mathcal{L}_A(\psi_k + 1/k)\|_{\infty}$ satisfies $\lim_{k \in \mathbb{N}} \lambda_k = \lambda_A$. Moreover, 
\[
\lambda_k \psi_k = \sum_{\sigma(y) = x}e^{A(y)}(\psi_k(y) + 1/k) \geq (\inf_{k \in \mathbb{N}}(\psi_k) + 1/k)e^{-\|A\|_{\infty}} \,.
\]

By the above, we obtain that $\lambda_k \inf_{k \in \mathbb{N}}(\psi_k) \geq (\inf_{k \in \mathbb{N}}(\psi_k) + 1/k)e^{-\|A\|_{\infty}}$. Therefore, we can conclude that $\lambda_k > e^{-\|A\|_{\infty}}$, for all $k \in \mathbb{N}$, which implies that $\lambda_A > 0$.  

On other hand, if there exists some point $x \in X_{m, \beta}$, such that, $\psi_A(x) = 0$, then,  it is true that 
\begin{equation}
0 = \lambda^n_A \psi_A(x) = \mathcal{L}^n_A(\psi_A)(x) = \sum_{\sigma^n(y) = x}e^{S_n A(y)} \psi_A(y) \,.
\label{density}
\end{equation}

In other words, $\psi_A(y) = 0$, for all $y \in \sigma^{-n}(\{x\})$. Since the quasi-greedy $\beta$-expansion of $1$ satisfies (\ref{irreducible-sequence}), it follows that $X_{m, \beta}$ is topologically transitive, which implies that the set $\{y : y \in \sigma^{-n}(\{x\})\}$ is dense in $X_{m, \beta}$. Then, by continuity of $\psi_A$, we can conclude that $\psi_A \equiv 0$, which is a contradiction taking into account that $\lambda_A > 0$. 

\medskip

The proofs that the  eigenvalue is simple and the other properties are similar to the ones in \cite{MR1085356}.

\end{proof}

\end{document}